\documentclass[12pt, reqno]{amsart}
\usepackage{amssymb}
\usepackage{graphicx}
\usepackage{hyperref}
\hypersetup{colorlinks=true,linkcolor=black,anchorcolor=black,citecolor=black}

\usepackage{appendix}
\usepackage{xcolor}
\usepackage{longtable}
\usepackage{float}
\usepackage{tikz}
\usepackage[all]{xy}
\DeclareFontFamily{U}{mathb}{\hyphenchar\font45}
\DeclareFontShape{U}{mathb}{m}{n}{
      <5> <6> <7> <8> <9> <10> gen * mathb
      <10.95> mathb10 <12> <14.4> <17.28> <20.74> <24.88> mathb12
      }{}
\DeclareSymbolFont{mathb}{U}{mathb}{m}{n}
\DeclareMathSymbol{\righttoleftarrow}{3}{mathb}{"FD}

\theoremstyle{plain}
\newtheorem{prop}{Proposition}[section]
\newtheorem{theo}[prop]{Theorem}
\newtheorem{coro}[prop]{Corollary}
\newtheorem{lemm}[prop]{Lemma}
\theoremstyle{remark}

\theoremstyle{definition}
\newtheorem{rema}[prop]{Remark}

\newtheorem{exam}[prop]{Example}
\numberwithin{equation}{section}

\def\cT{{\mathcal T}}

\def\disc{{\mathrm{disc}}}
\newcommand{\longdashrightarrow}{%
  \mathrel{\tikz[baseline=-0.55ex]
  \draw[dashed,->] (0,0)--(1.2,0);}}

\def\sD{{\mathsf D}}

\def\fC{{\mathfrak C}}
\def\fD{{\mathfrak D}}

\def\fS{{\mathfrak S}}

\def\fS{{\mathfrak S}}

\def\bA{{\mathbb A}}

\def\bP{{\mathbb P}}
\def\bQ{{\mathbb Q}}

\def\bZ{{\mathbb Z}}
\def\bR{{\mathbb R}}

\def\bC{{\mathbb C}}

\def\rH{{\mathrm H}}

\def\Pic{\mathrm{Pic}}

\def\Gal{\mathrm{Gal}}

\def\Gr{\mathrm{Gr}}
\def\Cox{\mathrm{Cox}}

\def\Hom{\mathrm{Hom}}

\def\lim{\mathrm{lim}}

\def\Spec{\mathrm{Spec}}

\title[Universal torsor and stable rationality]{Universal torsors over quartic del Pezzo surfaces and stable rationality}

\author[Yuri Tschinkel]{Yuri Tschinkel}
\address{
  Courant Institute,
  251 Mercer Street,
  New York, NY 10012, USA}
\address{Simons Foundation\\
160 Fifth Avenue\\
New York, NY 10010\\
USA}
\email{tschinkel@cims.nyu.edu}

\author[Zhijia Zhang]{Zhijia Zhang}

\address{
Department of Mathematics, Columbia University, New York, NY 10027, USA
}

\email{zz3531@columbia.edu}

\date{\today}

\makeatletter
\@namedef{subjclassname@2020}{\textup{2020} Mathematics Subject Classification}
\makeatother


\begin{document}

\begin{abstract}
Let $S$ be a smooth quartic del Pezzo surface over a field $k$, of characteristic zero. We prove that a universal torsor $\mathcal T$ over $S$ is $k$-rational, provided $\mathcal T$ has $k$-points.  
As an application, we obtain examples of stably rational smooth cubic hypersurfaces over $\mathbb Q$ in every dimension greater than 2.
\end{abstract}
	
	\maketitle

    \section{Introduction}
	\label{sect:intro}

Let $S$ be a del Pezzo surface over a field $k$ with a nonempty set 
$S(k)$ of $k$-rational points. 
Let 
$$
\cT\to S
$$
be a universal torsor over $S$, under the N\'eron-Severi torus of $S$; see Section~\ref{sect:back} for details on the terminology. 
Our first result is a proof of a conjecture of Colliot-Th\'el\`ene--Sansuc \cite[Section 2.8]{CTSansucDuke}, for quartic del Pezzo surfaces: 

\begin{theo}[Theorem~\ref{thm:mainrepeat}]
\label{thm:main}
Let $k$ be a field of characteristic zero and $S$ a quartic del Pezzo surface over $k$. Let $\cT\to S$ be a universal torsor over $S$ with $\cT(k)\neq \varnothing$. Then $\cT$ is $k$-rational.
\end{theo}

Rationality of universal torsors is a key property in the study of 
the Hasse principle and weak approximation on
del Pezzo surfaces over number fields, see, e.g., \cite{CTSansucDuke} and \cite{Salberger}. 
Here, we obtain applications to stable rationality:

\begin{coro}[Corollary~\ref{coro:dp4stab}]
\label{coro:stable}
Let $k$ be a field of characteristic zero and $S$ a quartic del Pezzo surface over $k$. 
Then $S$ is stably rational over $k$ if and only if 
\begin{itemize}
    \item $S(k)\neq \varnothing$ and
    \item $S$ satisfies Condition {\bf (H1)}, i.e., vanishing of Galois cohomology
    \begin{equation} 
    \label{eqn:h1}
\rH^1(\mathfrak g_{k'}, \Pic(\bar{S}))=0, \quad \text{ for all finite }  k'/k. 
    \end{equation}
\end{itemize}
\end{coro}

Rationality of del Pezzo surfaces is well-understood in terms of the Galois action on the geometric Picard group $\Pic(\bar{S})$,  see, e.g.,  \cite{manin-book}. 
For quartic del Pezzo surfaces $S$, the Galois action factors through
a subgroup of $W(\mathsf D_5)$, the Weyl group of the root system $\mathsf D_5$.  
There are 4 types of Galois actions that correspond to nonrational $S$ satisfying the conditions of Corollary~\ref{coro:stable}, see \cite[Theorems 4.19 and 5.20]{KSS}. Explicit presentations of these actions can be found in Section~\ref{sect:stab-rat} (see also \cite{TY}).
We list them by their images in $W(\mathsf D_5)$:
\begin{itemize}
    \item $I_0$ -- $\fS_3$,
    \item $I_1$ -- $\fC_2\times \fS_3$, 
    \item $I_2$ -- $\fC_3\rtimes \fC_4$, 
    \item $I_3$ -- $\fC_3\rtimes \fD_4$. 
\end{itemize}
Rationality of $\cT$ in the case $I_0$ was established in \cite[Section 1]{bctspstably}. The corresponding $S$ admits a conic bundle over $k$, and $\mathrm{rk} \Pic(S)=2$. 
The main application in \cite{bctspstably} produced the first example of a 
stably rational but nonrational threefold $X$, as a smooth projective model of a quartic del Pezzo surface over the function field $k=\bC(t)$, with Galois action of type $I_0$. Stable rationality of $X$ over $\bC$ follows formally from the stable rationality of $S$ over $k$. On the other hand, $X$ is also a conic bundle over a rational surface, and the intermediate Jacobian of $X$ is the Prym variety of the discriminant curve, with the associated double cover. Choosing an appropriate configuration of the discriminant curve, one can ensure that $X$ is nonrational. 

This beautiful construction served as input for \cite{HKT-duke} to produce smooth families of threefolds with stably rational and stably nonrational fibers. It seemed natural to try to apply it to other nonrational threefolds, e.g., cubic threefolds. This was sketched in unpublished notes \cite{HT-1}, \cite{HT-2}, where it was realized that: 
\begin{itemize}
    \item the $I_0$ case cannot arise from a smooth cubic threefold (birationally) as a quartic del Pezzo surface fibration,
    \item a proof of stable rationality in the $I_1$ case over $k=\bC(t)$ would produce stably rational smooth cubic threefolds over $\bC$. 
\end{itemize}
The bottleneck that remained was a proof of rationality of the torsor.  We settle this in Section~\ref{sect:rat}. 
In Section~\ref{sect:stab-rat}, we apply this to a proof of Corollary~\ref{coro:stable}. The proof follows exactly the strategy outlined in \cite[Section 1]{bctspstably}:
\begin{itemize}
\item stable rationality of the N\'eron-Severi torus $T$ of $S$ -- this follows from the fact that, in our situation, $\Pic(\bar{S})$ is a stably permutation module, 
\item existence of a universal torsor $\cT$ over $S$ with $\cT(k)\neq \varnothing$, and $k$-rationality of $\cT$, 
\item $k$-birationality of $\cT$ to $S\times T$, and thus stable rationality of $S$ over $k$. 
\end{itemize}

In Section~\ref{sect:cub}, we prove:
\begin{theo}[Propositions~\ref{prop:CubicI1} and~~\ref{prop:cubictypeI3}]
\label{thm:main3}
For every $n\ge 3$ there exist stably rational smooth cubic hypersurfaces $X\subset \bP^{n+1}$, over 
$\bQ$. 
\end{theo}
Combining Theorem~\ref{thm:main3} with the recent proof in \cite{cubicmatroid} that a very general cubic threefold over $\bC$ does not admit a decomposition of the
diagonal, and is thus not stably rational over $\bC$, we obtain new smooth families of threefolds with stably rational and stably nonrational fibers, extending the results of \cite{HKT-duke}.

The proof of Theorem~\ref{thm:main} was arrived at via a detour through equivariant geometry, that provided some rigidity in the quest for a rationality construction.  

The analogies between geometry over nonclosed fields and equivariant geometry have been discussed elsewhere, e.g., \cite{KT-survey}, \cite{KT-notions}. In particular, there is also a Condition {\bf (H1)}, which is necessary for stable linearizability of a regular action of a finite group $G$ on a del Pezzo surface over an algebraically closed field, identical to \eqref{eqn:h1}, but with Galois cohomology replaced by group cohomology:
$$
\rH^1(G',\Pic(S))=0, \quad \forall G'\subseteq G. 
$$
In the situation at hand, only the $I_2$-action can be realized as a geometric action on a del Pezzo surface over an algebraically closed field of characteristic zero, by \cite[Theorem 1.2]{Pro-2}. The following answers the last remaining open question in this direction: 

\begin{theo}\label{thm:main2}
    Let $S$ be the quartic del Pezzo surface given by 
    \begin{align*}
          S=\{x_1^2+\zeta_3x_2^2+\zeta_3^2x_3^2+x_4^2=x_1^2+\zeta_3^2x_2^2+\zeta_3x_3^2+x_5^2=0\}\subset\bP^4,
    \end{align*}
    with an action of $G=\fC_3\rtimes\fC_4$ generated by 
    \begin{align*}
        (x_1,\ldots,x_5)&\mapsto (x_2,x_3,x_1,\zeta_3x_4,\zeta_3^2x_5),\\
        (x_1,\ldots,x_5)&\mapsto (x_1,x_3,x_2,-x_5,x_4).
    \end{align*}
    Then the $G$-action on $S$ is stably linearizable. 
\end{theo}

We prove this in Section~\ref{sect:C3C4}, by writing down one more time a suitable universal torsor $\cT$, 
lifting the $G$-action to $\cT$, and proving the linearizability of $\cT$. The key observation is that $\cT$ has a $G$-fixed point. This is the first nontrivial example of a stably linearizable but not coarsely linearizable action on a rational variety, in the terminology of \cite{KT-notions}.

\ 

{\bf Acknowledgments:} We are grateful to A. Kresch and B. Hassett for many years of collaboration on this and related problems, and to J.-L. Colliot-Th\'el\`ene for his interest and suggestions.  
We used Claude Fable 5 and ChatGPT-5.6 Sol for searching the literature, testing hypotheses, checking computations, and proofreading. 
The paper was written by the authors, who are responsible for its correctness.

\section{Generalities}
\label{sect:back}

The material in this section is standard, and is assembled here for convenience. 
Throughout, $k$ is a field of characteristic zero. 
We fix an algebraic closure $\bar{k}$ of $k$ and let $\mathfrak g=\mathfrak g_k$ be the Galois group of $\bar{k}/k$. For a variety $X$ over $k$, we let $\bar{X}$ be the base change to $\bar{k}$. 

\subsection*{Modules}
Let $G$ be a profinite group, with continuous action on a discrete finite-rank $\bZ$-module $\mathsf M$, and 
$$
\rH^1(G, \mathsf M)
$$
the first cohomology group of $\mathsf M$.  
A $G$-module $\mathsf M$ is called a {\em permutation module} if it has a $\bZ$-basis permuted by $G$, and a {\em stably permutation module} if there exist permutation $G$-modules $\mathsf P$ and $\mathsf P'$, and an isomorphism of $G$-modules $\mathsf M\oplus \mathsf P\simeq \mathsf P'$. By Shapiro's lemma, if $\mathsf M$ is stably permutation, then $\rH^1(H,\mathsf M)=0$, for every open subgroup $H\subseteq G$. 

\subsection*{Tori}
There is a duality between the category of algebraic tori over $k$ and discrete continuous finite-rank $\mathfrak g$-modules, via the assignment
$$
T \mapsto \mathfrak X^*(\bar{T}), 
$$
sending $T$ to the group of characters of $\bar{T}$. Arithmetic properties of $T$ are reflected in properties of the $\mathfrak g$-module 
$\mathfrak X^*(\bar{T})$, e.g., 
if $\mathfrak X^*(\bar{T})$ is a stably permutation $\mathfrak g$-module, then $T$ is stably rational over $k$ \cite[Section 2]{CS-tori}.

\subsection*{Universal torsors}
Let $T$ be an algebraic torus over $k$ and $S$ a smooth projective variety over $k$, with $\Pic(\bar{S})$ torsion-free and of finite rank, e.g., a del Pezzo surface.  
There is a natural homomorphism 
$$
\rH^1(S,T) \stackrel{\chi}{\longrightarrow} \Hom_{\mathfrak g} (\mathfrak X^*(\bar{T}), \Pic(\bar{S}))
$$
from the group classifying isomorphism classes of $T$-torsors over $S$, 
see \cite[Theorem 1.5.1]{CTSansucDuke}. 
When $T$ is the N\'eron-Severi torus, i.e., 
$$
\mathfrak X^*(\bar{T})\simeq \Pic(\bar{S}),
$$
a $T$-torsor  $\cT$ is called {\em universal} if 
$$
\chi([\cT]) = \mathrm{Id}.
$$
Such torsors exist over $\bar{k}$ but not always over $k$. However, they do exist under the assumption that $S(k)\neq \varnothing$ \cite[Remark 2.2.9]{CTSansucDuke}.

\begin{lemm}\label{lemm:dense}
Let $S$ be a del Pezzo surface of degree $\ge 3$ over $k$ and $\cT$ a universal torsor over $S$. If $\cT(k)\neq \varnothing$ then $\cT(k)$ is Zariski dense in $\cT$.
\end{lemm}
\begin{proof}
Let $\xi:\cT\to S$ be the torsor and $t\in \cT(k)$.
Then $S(k)$ is nonempty, as it contains $s:=\xi(t)$, and $S$ is $k$-unirational, by
\cite[Theorems 29.4 and 30.1]{manin-book}. 
In degree $\ge 6$, universal torsors are Zariski open subsets of affine space. In degree 5, a universal torsor embeds as a dense Zariski open subset into the affine cone over the Grassmannian $\Gr(2,5)$, and rational points on it are dense as well. 

It remains to treat degree 3 and 4. The specialization $\cT_{s}$ of $\cT$ to $s$ is split, i.e., 
the class
$$
[\cT_{s}]\in \rH^1(\Spec(k), T)
$$
is trivial, since there is a point $t\in\cT_s(k)$. This implies that the fiber is isomorphic to $T$, and has Zariski dense $k$-rational points.  

By \cite[Proposition 2.7.2]{CTSansucDuke}, specializations of the universal torsor $\cT$ to $R$-equivalent points in $S(k)$ have the same class in $\rH^1(\Spec(k),T)$. It suffices to show that points in $S(k)$ which are $R$-equivalent to $s$ are Zariski dense in $S$. Indeed, for each $s'\in S(k)$ in the same $R$-equivalence class as $s$, the torsor $\cT_{s'}$ is split. Each such fiber has a Zariski dense set of rational points. Thus $k$-points are Zariski dense in $\cT$.  

For smooth cubic surfaces, this follows from 
\cite[Theorems 13.1 and 14.3]{manin-book}. 
The case of quartic del Pezzo surfaces $S$ can be reduced to the cubic surface case via a blowup of a general $k$-point (distinct from $s$ and away from exceptional curves), or
handled via a unirational parametrization $\mu: \bP^2\dashrightarrow S$ containing $s$ in its image: all points in the image of $\mu$ will be in the same $R$-equivalence class and Zariski dense. Arguments of this type can be found in \cite[Section 2.9]{CTSansucDuke}.
\end{proof}



\begin{rema}
\label{rema:split}
In our applications to stable rationality, the character lattice
$\mathfrak X^*(\bar{T})$ is a stably permutation $\mathfrak g$-module, with 
$$
\rH^1(\Spec(K),T)=0, 
$$
for any field $K/k$, by Hilbert's Theorem 90. In particular, the $T$-torsor $\cT\to S$ splits over the function field $k(S)$, and every fiber over a point $s\in S(k)$ is isomorphic to $T$, over $k$. 
\end{rema}

\subsection*{Stable rationality}
Recall that an algebraic variety $X$ is rational if $X$ is birational to $\bP^n$, and is stably rational if $X\times\bP^{r}$ is rational, for some integer $r$.
The application of the theory of universal torsors to 
stable rationality was pioneered in \cite{bctspstably}, which led to the first (and so far only) examples of stably rational but nonrational varieties. The stable rationality construction relies on the following:

\begin{prop}[{\cite[Proposition 3]{bctspstably}}]\label{prop:bctsmain}
    Let 
$X$ be a smooth, proper, geometrically integral, and geometrically rational 
variety over $k$. Assume that 
there exists a universal torsor 
$\cT\to X$ which is a
$k$-rational variety, and the 
$\mathfrak g$-module $\Pic(\bar{X})$ is a stably permutation module. Then $X$ is stably rational over $k$.
\end{prop}

The proof of Proposition~\ref{prop:bctsmain} uses Hilbert's Theorem 90 and Shapiro's lemma to establish a birationality 
$$
\cT\sim_k X\times T,
$$
where $T$ is the N\'eron-Severi torus of $X$.
Stable rationality of $X$ then follows from the assumptions on $\cT$ and $\Pic(\bar{X})$.

\subsection*{OADP varieties}
Recall that an $n$-dimensional reduced, irreducible, projective variety $X\subset \bP^{2n+1}$ has {\em one apparent double point} if through a general point in $\bP^{2n+1}$, there passes a unique secant line of $X$.
 See \cite{projection} for a discussion of properties of such varieties; rationality constructions based on apparent double points can be found in, e.g., \cite{CMR}, \cite[Sections 3 and 7]{HKT-quad}. 
Our proof of rationality of universal torsors $\cT$ over quartic del Pezzo surfaces will be based on the following strengthening of \cite[Theorem 2.7]{projection} (see also, \cite[Theorem 4.1]{CMR}):

\begin{theo}
\label{thm:projection}
Let $X\subset \bP^{2n+1}$ be an $n$-dimensional reduced, irreducible, nondegenerate projective variety over a field $k$ of characteristic zero. Assume that
$X(k)$ is Zariski dense and that $\bar{X}$ has one apparent double point. Then the projection 
$$
X\dashrightarrow \bP^n
$$
from the tangent space  
    of a general point $x\in X(k)$ is a birational map.
\end{theo}

\begin{proof}
The statement in \cite[Theorem 2.7]{projection} is formulated over $\bC$. But the proof works for any algebraically closed field $\bar k$ of characteristic 0. 
The extension to a nonclosed field $k$ is immediate: the projection map from the tangent space of a general $k$-point is defined over $k$, and is birational over $\bar{k}$. Thus it is birational over $k$. Zariski density of $X(k)$ allows to pick an appropriate point. 
\end{proof}

\section{Universal torsors over quartic del Pezzo surfaces}
\label{sect:rat}

We recall the construction of universal torsors over quartic del Pezzo surfaces. In the geometric case, this has been worked out in \cite{Coxderenthalhomogeneous}, \cite{batycox}, \cite{SS}. 
For results over nonclosed fields, see \cite{CTSansucDuke}, \cite{derenthal-non}.

We start with the geometric case, i.e., when $k=\bar{k}$. 
A universal torsor admits an embedding 
$$
\cT\hookrightarrow\Spec(\Cox(S)),
$$
as a dense Zariski open subset; 
here $\Cox(S)$ is the {\em Cox ring} of a quartic del Pezzo surface $S$. 
Concretely, let $S\to \bP^2$ be given as a blowup of 5 points in general position, which can be chosen to be
$$
p_1=[1:0:0],\quad p_2=[0:1:0],\quad p_3=[0:0:1], \quad  p_4=[1:1:1], 
$$
$$
p_5=[1:a:b],
$$
where $a,b\in k$ satisfy
$$
ab(a-1)(b-1)(a-b)\neq 0. 
$$
 The Cox ring $\Cox(S)$ is generated by 16 sections corresponding to exceptional curves in $S$ modulo 20 relations arising from 10 rulings in $S$, see \cite[Theorem 4.9]{batycox}. 
The 16 exceptional curves are 
\begin{itemize}
    \item 5 exceptional curves $E_i$ above $p_i$,
    \item  10 strict transforms $L_{ij}$ of lines  passing through $p_i$ and $p_j$,
    \item the strict transform $Q$ of the conic passing through the 5 points.
\end{itemize}
The corresponding generators of $\Cox(S)$ are denoted by 
$$
e_i,\quad l_{ij},\quad q.
$$
To each exceptional curve we assign a form defining its image in $\bP^2_{z_1,z_2,z_3}$ 
$$
f_{E_i}=1,\quad f_{L_{12}}=z_3,\quad f_{L_{13}}=z_2,\quad f_{L_{14}}=z_2-z_3,\quad f_{L_{15}}=bz_2-az_3,
$$   
$$
f_{L_{23}}=z_1,\quad f_{L_{24}}=z_1-z_3,\quad f_{L_{25}}=bz_1-z_3,\quad f_{L_{34}}=z_1-z_2,
$$
$$
f_{L_{35}}=az_1-z_2,\quad f_{L_{45}}=(b-a)z_1+(1-b)z_2+(a-1)z_3,
$$
$$
f_Q=b(1-a)z_1z_2+a(b-1)z_1z_3+(a-b)z_2z_3.
$$
 These determine relations in $\Cox(S)$. The 10 rulings in $S$ are classes
$$
H-E_i,\quad 2H+E_i-\sum_{j=1}^5E_j,\quad i=1,\ldots,5,
$$
where $H$ is the pullback of a general line on $\bP^2$.
We compute the relations associated with each ruling: For each $i=1,\ldots,5$, there are 4 reducible members in the pencil $|H-E_i|$, 
    $$
    L_{ij}+E_j, \quad j\in\{1,\ldots,5\}\setminus\{ i\},
    $$
yielding relations given by the vanishing of the following forms:
    \begin{align*}
       R_1={}&e_2l_{12}-e_3l_{13}+e_4l_{14},\\
R_2={}&a e_2l_{12}-b e_3l_{13}+e_5l_{15},\\
R_3={}&e_1l_{12}-e_3l_{23}+e_4l_{24},\\
R_4={}&e_1l_{12}-b e_3l_{23}+e_5l_{25},\\
R_5={}&e_1l_{13}-e_2l_{23}+e_4l_{34},\\
R_6={}&e_1l_{13}-a e_2l_{23}+e_5l_{35},\\
R_7={}&e_1l_{14}-e_2l_{24}+e_3l_{34},\\
R_8={}&(b-1)e_1l_{14}+(a-b)e_2l_{24}+e_5l_{45},\\
R_9={}&e_1l_{15}-a e_2l_{25}+b e_3l_{35},\\
R_{10}={}&(a-1)e_2l_{25}+(1-b)e_3l_{35}+e_4l_{45}.
    \end{align*}
For  $i=1,\ldots,5$, the  4 reducible members in $ |2H+E_i-\sum_{j=1}^5E_j|$ are 
    $$
    Q+E_i,\quad L_{i_1i_2}+L_{i_3i_4}, \quad \{i_1,i_2,i_3,i_4\}=\{1,\ldots,5\}\setminus \{i\}.
    $$
The remaining relations are given by the vanishing of the following forms:
\begin{align*}
R_{11}={}&l_{23}l_{45}+l_{24}l_{35}-l_{25}l_{34},\\
R_{12}={}&a l_{23}l_{45}+(a-b)l_{24}l_{35}-e_1q,\\
R_{13}={}&l_{13}l_{45}+l_{14}l_{35}-l_{15}l_{34},\\
R_{14}={}&l_{13}l_{45}+(1-b)l_{14}l_{35}-e_2q,\\
R_{15}={}&l_{12}l_{45}+l_{14}l_{25}-l_{15}l_{24},\\
R_{16}={}&l_{12}l_{45}+(1-a)l_{14}l_{25}-e_3q,\\
R_{17}={}&l_{12}l_{35}-l_{13}l_{25}+l_{15}l_{23},\\
R_{18}={}&(b-1)l_{12}l_{35}+(1-a)l_{13}l_{25}-e_4q,\\
R_{19}={}&l_{12}l_{34}-l_{13}l_{24}+l_{14}l_{23},\\
R_{20}={}&a(b-1)l_{12}l_{34}+b(1-a)l_{13}l_{24}-e_5q.
\end{align*}   
Then 
$$
 \Cox(S)=k[q,e_1,\ldots,e_5,l_{ij}]/\langle R_1,\ldots, R_{20}\rangle, \quad 1\leq i<j\leq 5,
$$
and
\begin{align}\label{eqn:coxring}
U:=\Spec(\Cox(S))\subset \bA^{16}
\end{align}
contains a universal torsor $\cT$ over $S$ as a dense Zariski open subset. 
Taking the projectivization gives $\bP(U)\subset\bP^{15}$. Let 
\begin{align}\label{eqn:Jacobian}
J:=\left(\frac{\partial R_i}{\partial x_j}\right)
\end{align}
be the Jacobian matrix, where $x_j$ is the $j$th generator of $\Cox(S)$, with the natural lexicographic order on $l_{ij}$.
\begin{lemm}\label{lemm:Jacobian}
Let $u_1,u_2$ be two points on $U$. Then
$$
u_1-u_2\in \ker(J(u_1+u_2)).
$$
\end{lemm}
\begin{proof}
    Let $B_i(x,y)=R_i(x+y)-R_i(x)-R_i(y)$ be the symmetric bilinear form associated with the quadratic form $R_i$.  For any $t\in k$, we have that
    \begin{align*}
        R_i(x+ty)=tB_i(x,y)+R_i(x)+t^2R_i(y).
    \end{align*}
     The $i$th component of $J(x)\cdot y$ is
    $$
   (J(x)\cdot y)_i=\left.\frac{\mathrm d}{\mathrm dt}\right\vert_{t=0}R_i(x+ty)=B_i(x,y).
    $$
    It follows that for every $i=1,\ldots,20$, 
    \begin{align*}
        \left(J(u_1+u_2)\cdot (u_1-u_2)\right)_i=B_i(u_1+u_2,u_1-u_2)=2R_i(u_1)-2R_i(u_2)=0.
    \end{align*}
\end{proof}

The key observation is that $\bP(U)$ is an OADP variety:
\begin{lemm}
\label{lemm:secant}
   Through a general point in $\bP^{15}$, there passes a unique secant line of $\bP(U)$. 
\end{lemm}
\begin{proof}
  First, we show that the secant variety of $\bP(U)$ is the whole $\bP^{15}$. By Terracini’s Lemma (see, e.g., \cite[Theorem 1.1]{projection}), it suffices to show that for two general points $u_1,u_2\in U$, we have that
  \begin{align}\label{eqn:terracini}
      T_{u_1}(U)+T_{u_2}(U)=\bA^{16},
  \end{align}
   where $T_{u_i}(U)$ is the embedded tangent space of $U$ at $u_i$, $i=1,2.$ 
 Consider 
 $$
t_1=(0,1,1,0,\ldots,0),\quad t_2=(1,0,0,0,0,0,1,0,\ldots,0).
$$
A computation\footnote{See \cite{TZdp4-tables} for supporting materials.} of the Jacobian matrix $J$ defined in~\eqref{eqn:Jacobian}  shows that $t_1,t_2$ are smooth points of $U$ and $
  T_{t_1}(U)+T_{t_2}(U)=\bA^{16}.
  $
Since~\eqref{eqn:terracini} is an open condition, it holds for  two general points on $U$.

Let $x\in \bA^{16}$ be a general point. We now show that there is exactly one secant line of $\bP(U)$ passing through the image of $x$ in $\bP^{15}$.
 Assume that  two such lines $l_1$ and $l_2$ exist. We can find points $t_3,t_4,t_3',t_4'$ on $U$ such that
$
l_1=\bP(\langle t_3,t_4\rangle),$ $ l_2=\bP(\langle t_3', t_4'\rangle),
$ and $x=t_3+t_4=t_3'+t_4'$.
By Lemma~\ref{lemm:Jacobian}, we know that 
$$
J(x)\cdot (t_3-t_4)=J(x)\cdot(t_3'-t_4')=0.
$$
A direct computation shows that 
$
   \dim(\ker(J(t_1+t_2)))=1,
$
which implies that 
$
   \dim(\ker(J(x)))\leq 1,
$
since this is an open condition and $x$ is general.
It follows  that $t_3-t_4$ and $t_3'-t_4'$ are linearly dependent, i.e., $l_1=l_2.$
\end{proof}

\begin{rema}
In any concrete situation, one can directly check the OADP property of $\bP(U)$ and the birationality of the projection in Theorem~\ref{thm:projection}. An example of this computation is provided in Section~\ref{sect:C3C4}.
\end{rema}

\begin{theo}\label{thm:mainrepeat}
    Let $S$ be a quartic del Pezzo surface over $k$ and $\cT\to S$ a universal torsor over $S$. If $\cT(k)\neq \varnothing$, then  $\cT$ is rational over $k$. 
\end{theo}

\begin{proof}
By Lemma~\ref{lemm:dense}, $\cT(k)$ is Zariski dense in $\cT$.
The $\mathfrak g$-equivariant structures
\begin{align}\label{eqn:torsorembed}
    \bar{\cT}\hookrightarrow \bar{U}:=\Spec(\Cox(\bar{S}))\subset \bA^{16}
\end{align}
descend to $k$, i.e., we have 
an embedding
\begin{align}\label{eqn:torsordescend}
\cT\hookrightarrow U\subset \bA^{16},
\end{align}
over $k$, such that~\eqref{eqn:torsorembed} is the base change of~\eqref{eqn:torsordescend} to $\bar k$, see, e.g., \cite[Remark 4.5]{derenthal-non}.
On passage to $\bar{k}$, the universal torsor $\bar\cT$ with the embedding~\eqref{eqn:torsorembed} is linearly isomorphic to the embedding of the universal torsor constructed in~\eqref{eqn:coxring}. Geometrically, $\bar{\cT}$ is a Zariski dense open subset of a cone over a variety $\bP(\bar{U})$ which 
has one apparent double point by Lemma~\ref{lemm:secant}. Note that $\bP(\bar{U})$ is reduced, irreducible and nondegenerate, since $\Cox(\bar S)$ is an integral domain and there are no linear forms in the defining ideal. 
By Theorem~\ref{thm:projection}, projection from the tangent space of a general $k$-point in $\bP(U)$ yields $k$-rationality of $\bP(U)$ and thus of $\cT$. 
\end{proof}

\begin{coro}\label{coro:rattorsor}
    Let $S$ be a quartic del Pezzo surface over $k$ such that $S(k)\ne \varnothing$. Then there exists a $k$-rational universal torsor $\cT\to S$.
\end{coro}
\begin{proof}
    By Theorem~\ref{thm:mainrepeat}, it suffices to show the existence of a universal torsor $\cT\to S$ with a rational point. Recall that $S(k)\ne \varnothing$ implies that we can find $s\in S(k)$ in the complement to exceptional curves. By \cite[Corollary 2.3.4]{CTSansucDuke}, there exists a universal torsor $\cT$ over $S$ whose fiber over $s$ is a trivial torsor, and thus $\cT(k)\ne \varnothing$.
\end{proof}

\section{Stable rationality of quartic del Pezzo surfaces}
\label{sect:stab-rat}

We recall basic facts about the Picard group of a quartic del Pezzo surface $S$, see \cite{manin-book} and~\cite{KSS}. The anticanonical model of $S$ is a smooth intersection of two quadrics 
$$
S=Q_1\cap Q_2\subset \bP^4.
$$
The pencil of quadrics generated by $Q_1$ and $Q_2$ has 5 singular members, over $\bar{k}$. Each of them is a cone over $\bP^1\times\bP^1$, whose two rulings give rise to two conic bundle structures on $\bar{S}$. Choosing an appropriate contraction $\bar{S}\to \bP^2$ and labeling the 16 exceptional curves in $\bar{S}$ by $Q, E_i, L_{ij}$ as in Section~\ref{sect:rat},  the classes of 10 rulings in $\Pic(\bar{S})$ are
\begin{align}
    \label{eqn:rulings}
    F_i:=H-E_i,\quad F_i'=2H+E_i-\sum_{j=1}^5E_j,\quad i=1,\ldots,5,
\end{align}
where $H$ is the pullback of a general line on $\bP^2$. Each pair $\{F_i,F_i'\}$ consists of the two rulings of the same singular quadric.

The $\mathfrak g$-action on $\Pic(\bar{S})$ factors through a subgroup of the Weyl group 
$$
W(\sD_5)\simeq \fC_2^4\rtimes\fS_5,
$$
where $\fS_5$ permutes the 5 singular members of the pencil, and $\fC_2^4$ is generated by elements simultaneously switching the rulings of the $i$th and $j$th singular members, denoted by $c_ic_j$ . 

Choose a bijection between the 16 exceptional curves, and the 16 odd subsets of $\{1,\ldots,5\}$ given by
$$
Q\leftrightarrow\{1,\ldots,5\},\quad L_{ij}\leftrightarrow\{1,\ldots,5\}\setminus\{i,j\},\quad E_i\leftrightarrow\{i\}.
$$
The $W(\sD_5)$-action on exceptional curves can be described in terms of its action on odd subsets of $\{1,\ldots,5\}$ as follows: $\fS_5$ permutes the 5 indices, and $c_ic_j$ changes an odd subset $I$ to $J$ where for any $r\in\{1,\ldots,5\}$,
\begin{itemize}
    \item if $r\in I$, then $r\in J$ if and only if $r\notin \{i,j\}$,
    \item if $r\notin I$, then $r\in J$ if and only if $r\in \{i,j\}$.
\end{itemize}

Here is an alternative way to identify the image of $\mathfrak g$ in $W(\sD_5)$, using the blowup $\tilde{S}\to S$ of a general point $s\in S(k)$ (this will be used in Section~\ref{sect:cub}). Note that $\tilde{S}$ is a cubic surface with a line $l$ above $s$. We denote the class of $l$ in the geometric Picard group of $\tilde{S}$
by $E_6$. The residual conic bundle 
\begin{align}
    \label{eqn:cubconicbundle}
    \varphi:\tilde{S}\to\bP^1
\end{align}
has, geometrically, 5 singular fibers, each consisting of two distinct lines. Their components are the 10 lines intersecting $l$. The $i$th singular fiber consists of two lines whose classes are
 $$
 H-E_i-E_6,\quad 2H+E_i-\sum_{j=1}^6E_j,\quad i=1,\ldots,5.
 $$
Their images in $S$ are exactly the 10 rulings in~\eqref{eqn:rulings}. Therefore, the image of $\mathfrak g\to W(\sD_5)$ is also determined by the permutation of 5 singular fibers  of $\varphi$ and an even number of switches of their components.

In \cite{Maninperfect},  Manin shows 
that for a stably rational surface $S$ over a field $k$ of characteristic zero, the $\mathfrak g$-module $\Pic(\bar{S})$ is stably permutation. We are interested in the converse.

\begin{prop}[{\cite{KSS}, \cite{TY}}]\label{prop:h1}
    If a quartic del Pezzo surface $S$ is $k$-minimal, and satisfies Condition {\bf (H1)} of \eqref{eqn:h1}, i.e., 
    $$
    \rH^1(\mathfrak g_{k'},\Pic(\bar{S}))=0,\quad \text{for all finite } k'/k,
    $$
   then the image of $\mathfrak g_{k}$ in $W(\sD_5)$ is conjugate to one of the following:
   \begin{itemize}
       \item[$(I_0):$] $\fS_3=\langle (3,4,5),c_1c_3c_4c_5(4,5)\rangle$,
         \item[$(I_1):$] $\fC_2\times\fS_3=\langle c_2c_3(4,5),c_1c_2c_3c_4(3,4,5)\rangle$,
           \item[$(I_2):$] $\fC_3\rtimes\fC_4=\langle (1,2,3),c_1c_2c_3c_4(2,3)(4,5)\rangle$,
             \item[$(I_3):$] $\fC_3\rtimes\fD_4=\langle c_1c_2c_3c_5(2,5),c_3c_4(3,4)(1,5,2)\rangle$.
   \end{itemize}
   \end{prop}
Our presentation of $I_1, I_2$, and $I_3$ is different from that in \cite[Theorem 5.20]{KSS}, but the resulting subgroups are conjugate  in $W(\sD_5).$

\begin{lemm}\label{lemm:stablyperm}
    Let $S$ be a quartic del Pezzo surface over $k$. Then 
    the $\mathfrak g$-module 
    $\Pic(\bar{S})$ is stably permutation if and only if Condition {\bf (H1)} of  \eqref{eqn:h1} is satisfied.
\end{lemm}
   \begin{proof}
      If $S$ is not $k$-minimal, the assertion follows from 
      \cite{manin-book}. Thus we may assume that $S$ is $k$-minimal. It suffices to show that $\Pic(\bar{S})$ is stably permutation, for the four types in Proposition~\ref{prop:h1}. The type $I_0$ is addressed in \cite{bctspstably}. Up to conjugation in $W(\sD_5)$, the groups of type $I_0$, $I_1$ and $I_2$ are subgroups of the group of type $I_3$. Thus we only need to consider this last type. 
      Let 
      $$
      g_1=c_1c_2c_3c_5(2,5), \quad g_2=c_3c_4(3,4)(1,5,2).
      $$
      In the standard basis $H, E_1,\ldots,E_5$ of  $\Pic(\bar{S})$, they act by matrices
{\small
$$
g_1=
\begin{pmatrix}
 3& 1& 1& 1& 2& 1\\
-1&-1& 0& 0&-1& 0\\
-1& 0& 0& 0&-1&-1\\
-1& 0& 0&-1&-1& 0\\
-2&-1&-1&-1&-1&-1\\
-1& 0&-1& 0&-1& 0
\end{pmatrix},
\quad
g_2=
\begin{pmatrix}
 2& 1& 1& 0& 0& 1\\
-1&-1& 0& 0& 0&-1\\
-1&-1&-1& 0& 0& 0\\
 0& 0& 0& 1& 0& 0\\
 0& 0& 0& 0& 1& 0\\
-1& 0&-1& 0& 0&-1
\end{pmatrix}.
$$}
In our convention, the matrices act from the right, e.g.,  
$$
g_2(H)=2H-E_1-E_2-E_5.
$$
We introduce new variables $w_0,w_1,w_2,q_0,q_1$ with $\mathfrak g$-permutation actions generated by 
$$
g_1(w_0)=w_0,\quad g_1(w_1)=w_2,\quad g_1(w_2)=w_1,\quad g_1(q_0)=q_1,\quad g_1(q_1)=q_0,
$$
$$
g_2(w_0)=w_2,\quad g_2(w_1)=w_0,\quad g_2(w_2)=w_1,\quad g_2(q_0)=q_0,\quad g_2(q_1)=q_1.
$$
One can check that 
$$
\Pic(\bar{S})\oplus\bZ[w_0,w_1,w_2,q_0,q_1]
$$ 
is a permutation module. Indeed, in the basis
\begin{align*}
b_1={}&
  2H-E_2-E_3-E_4-E_5-w_0-w_2-2q_0-q_1,\\
b_2={}&
  H-E_1-w_0-w_1-q_0-2q_1,\\
b_3={}&
  2H-E_1-E_3-E_4-E_5-w_0-w_1-2q_0-q_1,\\
b_4={}&
  H-E_2-w_1-w_2-q_0-2q_1,\\
b_5={}&
  H-E_5-w_0-w_2-q_0-2q_1,\\
b_6={}&
  2H-E_1-E_2-E_3-E_4-w_1-w_2-2q_0-q_1,\\
b_7={}&
  -3H+E_1+E_2+2E_3+E_4+E_5
  +w_0+w_1+w_2+3q_0+q_1,\\
b_8={}&
  -2H+E_1+E_2+E_5
  +w_0+w_1+w_2+q_0+3q_1,\\
b_9={}&
  -H+w_0+w_1+w_2+q_0+3q_1,\\
b_{10}={}&
  -3H+E_1+E_2+E_3+2E_4+E_5
  +w_0+w_1+w_2+3q_0+q_1,\\
b_{11}={}&
  3H-E_1-E_2-E_3-E_4-E_5
  -w_0-w_1-w_2-2q_0-2q_1,
\end{align*}
the permutation action of $\mathfrak g$ is visible by expressing
$$
g_1=(b_1,b_2)(b_3,b_5)(b_4,b_6)
  (b_7,b_8)(b_9,b_{10}),
  $$
  $$
  g_2=(b_1,b_6,b_3)(b_2,b_5,b_4)(b_8,b_9).
$$
\end{proof}
\begin{coro}\label{coro:dp4stab}
    Let $S$ be a quartic del Pezzo surface over a field $k$ of characteristic zero such that $S(k)\ne \varnothing$. Then $S$ is stably rational over $k$ if and only if one of the following equivalent conditions holds:
    \begin{enumerate}
           \item $\Pic(\bar{S})$ is a stably permutation $\mathfrak g$-module,
           \item $S$ satisfies Condition {\bf(H1)} of (1.1),
\item $S$ is not $k$-minimal or the $\mathfrak g$-action on $\Pic(\bar{S})$ factors through one of the four subgroups in Proposition~\ref{prop:h1}.
     \end{enumerate}
\end{coro}
\begin{proof}
    The equivalence of the conditions follows from Proposition~\ref{prop:h1} and Lemma~\ref{lemm:stablyperm}. Corollary~\ref{coro:rattorsor} and Proposition~\ref{prop:bctsmain} show that (1) implies the stable rationality of $S$. The converse is in \cite{Maninperfect}.
\end{proof}

\begin{rema}
    Examples of quartic del Pezzo surfaces over $\bQ$ of type $I_1, I_2$, or $I_3$ are given in \cite[Theorem 6.15]{KSS}; 
    Corollary~\ref{coro:dp4stab} yields new  
    stably rational but nonrational surfaces over nonclosed fields. 
\end{rema}

 \subsection*{Levels of stable rationality}
Let $S$ be a stably rational but nonrational quartic del Pezzo surface. What is the smallest integer $r$ such that $S\times\bA^r$ is rational?  
    Our construction of $\Pic(\bar{S})$ as a stably permutation module in Lemma~\ref{lemm:stablyperm} gives a uniform upper bound: $S\times\bA^{11}$ is rational. One can check that 5 and 11 are the smallest dimensions of $I_3$-permutation modules $\mathsf P$ and $\mathsf P'$ such that $\Pic(\bar{S})\oplus\mathsf P\simeq \mathsf P'$.   
    For $S$ of type $I_0$, \cite{bctspstably} shows that $r\leq 3$ by working with a torsor under a 3-dimensional torus, which was improved to $r\leq 2$ in \cite{SBstab}. It remains open to construct a nonrational variety $X$ such that $X\times\bA^1$ is rational.

\section{Cubic hypersurfaces}
\label{sect:cub}

In this section, we provide examples of stably rational smooth cubic hypersurfaces over $\bQ$ in every dimension $\ge 3$. No such examples in odd dimensions were previously known. Over $\bC$,  smooth cubic threefolds  are  nonrational, by \cite{CG}.  By \cite{cubicmatroid}, a very general cubic threefold over $\bC$ is not stably rational. After \cite{Voisincubic}, one might expect that stably rational cubic threefolds are dense in moduli -- we are unable to reach this conclusion with our constructions.

Let $X\subset\bP^{n+1}$ be a smooth cubic hypersurface with $n\geq3$, and $l$ a line contained in $X$. Choose a general plane $\Pi\simeq\bP^{2}\subset\bP^{n+1}$ containing $l$ and consider the projection from $\Pi$:
$$
X\dashrightarrow\bP^{n-2}.
$$
The generic fiber $X_{\eta}$ of the projection is a cubic surface over $K=k(\bP^{n-2})$ containing a line $l_\eta$ defined over $K$. Contracting $l_\eta$ yields  a quartic del Pezzo surface 
$S$ birational to $X_\eta$ over $K$, with a $K$-point. 

Our goal is to find appropriate $X,l,$ and $\Pi$ such that the $\mathfrak g_K$-action on $\Pic(\bar S)$ has one of the types $I_1$, $I_2$, or $I_3$. For these, Corollary~\ref{coro:dp4stab} shows that $S$ is stably rational over $K$, and thus $X$ is stably rational over $k$.

\subsection*{Type $I_1$} Here, we construct a smooth cubic hypersurface in every dimension $\ge 3$ which is birational to a quartic del Pezzo surface fibration of type $I_1$, and is stably rational over $\bQ$.
\begin{prop}\label{prop:CubicI1}
The smooth cubic hypersurface 
$$
X:=\{(x_4-2x_3)x_1^2+3(x_4+2x_3)x_2^2+3x_3^2x_4-x_4^3 +x_5^3 +\sum_{i=1}^rw_i^3=0\}\subset\bP^{r+4}
$$
is stably rational over $k=\bQ$, for every $r\geq0$.
\end{prop}
\begin{proof}
When $r=0$, one can  directly check that $X$ is smooth.    Adding the Fermat cubic in variables $w_i$ preserves the smoothness for every $r\geq 1$. 
Choose the line and the point
$$
l=\{x_3=x_4=x_5=w_1=\ldots=w_r=0\}, \quad p=[0:0:0:1:0:\ldots:0].
$$ 
Projection from $\Pi=\langle l,p\rangle$ is given by 
$$
X\dashrightarrow\bP^{r+1},\quad (x_i,w_j)\mapsto (x_3,x_5,w_1,\ldots,w_r).
$$
Let $$K=k(\bP^{r+1})=k(a,b_1,\ldots,b_r).$$ We work on the affine chart $x_5\ne 0$. Set $a={x_3}/{x_5},$ and $b_i={w_i}/{x_5}$. Renaming the coordinates 
$$
Y_1=x_1,\quad Y_2=x_2,\quad Y_3=x_5,\quad Y_4=x_4,
$$
the generic fiber $X_\eta\subset\bP^3_{Y_1,Y_2,Y_3,Y_4}$ is given by
\begin{align}\label{eqn:I1Xeta}
Y_3(\beta Y_3^2+2a(3Y_2^2-Y_1^2))+Y_4(Y_1^2+3Y_2^2+3a^2Y_3^2-Y_4^2)=0,
\end{align}
where 
$\beta=1+\sum_{i=1}^rb_i^3.
$
 One can check that $X_\eta$ is smooth over $K$. Contracting the line 
 $$
 l_\eta=\{Y_3=Y_4=0\}\subset\bP^3,
 $$
 we obtain a quartic del Pezzo surface $S\subset\bP^4_{Y_1,\ldots,Y_5}$ over $K$ given by 
$$
     \beta Y_3^2+2a(3Y_2^2-Y_1^2)-Y_4Y_5=Y_1^2+3Y_2^2+3a^2Y_3^2-Y_4^2+Y_3Y_5=0.
$$
We show that $S$ is stably rational over $K$.
Consider  the projection from the line $l_\eta$:
$$
\pi_\eta:X_\eta\to\bP^1,\quad (Y_1,Y_2,Y_3,Y_4)\mapsto(Y_3,Y_4).
$$ 
Over the affine chart $t={Y_4}/{Y_3}$, the generic fiber is the conic given by 
$$
f(Y_1,Y_2):=(t - 2a)Y_1^2 + 3(t + 2a)Y_2^2 - t^3 + 3a^2t + \beta=0.
$$
The discriminant of the quadratic form $f$ is
$$
\mathrm{disc}(f)=3(t-2a)(t+2a)(t^3-3a^2t-\beta).
$$
Let $\rho_1,\rho_2,\rho_3$ be the roots of the irreducible polynomial $$
c(t) = t^3 - 3a^2t -\beta\in K[t].
$$ 
A discriminant computation shows that
$$
\Delta:=\mathrm{disc}(c(t))=27(2a^3+\beta)(2a^3-\beta)
$$
is not a square in $K$. Thus $\Gal(K_1/K)=\fS_3$, 
where
$
K_1=K(\rho_1,\rho_2,\rho_3).
$
The five singular fibers of $\pi_\eta$ are individually defined over $K_1$. Put 
$$
d_1=\rho_2-\rho_3,\quad d_2=\rho_3-\rho_1,\quad d_3=\rho_1-\rho_2.
$$
The coefficients of $c(t)$ imply that
\begin{align*}
\rho_1+\rho_2+\rho_3=0,\quad \rho_1\rho_2+\rho_2\rho_3+\rho_1\rho_3=-3a^2,
\end{align*}
and thus
$$
3(4a^2-\rho_i^2)=d_i^2\in K_1^{\times2},\quad i=1,2,3.
$$
It follows that singular fibers above $\rho_1,\rho_2,$ and $\rho_3$ split over $K_1.$ We label their components by
\begin{align}\label{eqn:I1fiber1}
F_i^\pm: (\rho_i-2a)Y_1\pm d_iY_2=0.
\end{align}
On the other hand, singular fibers over 
$
2a$ and $-2a$  split after adjoining $e_1$ and $e_2$, respectively, where 
$$
e_1^2={3a(2a^3-\beta)},\quad e_2^2={a(2a^3+\beta)}.
$$ 
We label their components by
\begin{align}\label{eqn:I1fiber2}
 F_4^\pm:6aY_2\pm e_1=0,\quad F_5^\pm :2aY_1\pm e_2=0.
\end{align}
Observe that $e_2\in K_1(e_1)$. Indeed, one can check that
$$
\left(\frac{e_1}{9(2a^3-\beta)}\cdot d_1d_2d_3\right)^2=\left(\frac{e_1}{9(2a^3-\beta)}\right)^2\cdot\Delta=a(2a^3+\beta).
$$
 All components of the singular fibers of $\pi_\eta$ are individually defined over 
$$
K_1(e_1)=K(\rho_1,\rho_2,\rho_3,e_1),
$$
which is thus the splitting field of $X_{\eta}$ and $S$ over $K$. There is a unique quadratic subextension  $K(\sqrt{\Delta})/K$ of $K_1/K$. Assume that $e_1\in K_1$.  Then $K(\sqrt{\Delta})=K(e_1)$, implying that either  $e_1^2$ or $e_1^2/\Delta$ is a square in $K$. Computing their valuations at $a$, we see that this is impossible. Thus $e_1\not\in K_1$ and
$K(e_1)\cap K_1=K$. Since $K(e_1)K_1=K_1(e_1)$, we know that
$$
\Gal(K_1(e_1)/K)=\Gal(K(e_1)/K)\times \Gal(K_1/K)=\fC_2\times\fS_3.
$$
 Recall from Section~\ref{sect:stab-rat} that  the image of $\Gal(K_1(e_1)/K)$ (or equivalently, the image of $\mathfrak g_K$) in $W(\sD_5)$
is determined by its action on the 10 singular fibers \eqref{eqn:I1fiber1} and \eqref{eqn:I1fiber2}. \begin{itemize}
    \item The (untwisted) 3-cycle from $\fS_3$ permutes $\rho_i$ and thus $d_i$, and it fixes $e_1$ and $e_2$.  It maps to $(1,2,3)\in W(\sD_5)$. 
    \item Choose the (untwisted) 2-cycle $\tau$ from $\fS_3$ which fixes $\rho_1$ and $e_1$, and switches $\rho_2$ and $\rho_3$. Then $\tau(d_1)=-d_1$, $\tau(d_2)=-d_3$, $\tau(d_3)=-d_2$, and  $\tau(e_2)=-e_2$. It maps to $c_1c_2c_3c_5(2,3)$.
    \item Let $\iota\in \Gal(K_1(e_1)/K_1)$ be the order 2 element. We have that $\iota(d_i)=d_i,$ for $i=1,2,3$, and $\iota(e_i)=-e_i, $ for $i=1,2.$ It maps to $c_4c_5\in W(\sD_5)$.
\end{itemize}
Therefore, the image of $\mathfrak g_K$ in $W(\sD_5)$ is 
$$
\langle (1,2,3),c_1c_2c_3c_5(2,3),c_4c_5\rangle\subset W(\sD_5),
$$
which is conjugate to the group of type $I_1$ in Proposition~\ref{prop:h1}. Applying Corollary~\ref{coro:dp4stab}, we conclude that $S$ is stably rational over $K$. It follows that the smooth cubic $X$ is stably rational over $k=\bQ$.
\end{proof}

The same arguments yield stably rational but nonrational smooth intersections of two quadrics in $\bP^5$, over $\bQ$ and over $\bC(t)$. The following example was found using ChatGPT 5.6-Sol. 

\begin{exam}
\label{exam:x22}
Let $X=\{q_1=q_2=0\}\subset \bP^5$ be the smooth complete intersection of two quadrics, over $\bQ$, where
\begin{align*}
q_1&=2x_1^2-6x_2^2+3x_3^2-3x_3x_4+x_5^2-x_4x_6,
\\
q_2&=x_1^2+3x_2^2-x_4^2-x_3x_6.
\end{align*}
Then 
\begin{enumerate}
\item $X$ contains no lines over $\bQ$. In fact, $X$ contains no lines over $\bR$, since the quadratic form $q_1-\frac{21}{10}q_2$ has signature $(1,5)$. It follows that $X$ is not rational over $\bQ$ (and over $\bR$), by \cite{HT-quad} or \cite{BW}. 
\item The generic fiber of the rational map
\begin{equation*}
         f:X\dashrightarrow\bP^1,
         \qquad (x_1,\ldots,x_6)\mapsto(x_3,x_5),
\end{equation*}
is a quartic del Pezzo surface $S$ over $\bQ(a)$ which is isomorphic to the surface~\eqref{eqn:I1Xeta} with $\beta=-3a^3-a$. The same computation as above shows that $S$ is of type $I_1$, and thus $X$ is stably rational, over $\bQ$. 
\end{enumerate}
\end{exam}

\subsection*{Type $I_3$} Here, we construct a smooth cubic hypersurface over $\bQ$, in every dimension $\ge 3$, that is birational to a quartic del Pezzo surface fibration of type $I_3$, and is stably rational over $\bQ$.

\begin{prop}
    \label{prop:cubictypeI3}
    The smooth cubic hypersurface $X\subset\bP^{r+4}$ given by
    $$
    x_4(x_1^2+2x_1x_2)+x_3(x_1^2+x_1x_2+x_2^2)+x_3^3-x_3x_4^2+x_4^3+2x_5^3+\sum_{i=1}^rw_i^3=0
    $$
    is stably rational over $k=\bQ$, for any $r\geq 0.$
\end{prop}
\begin{proof}
  One can directly check that
$X$ is smooth when $r=0$. Adding the Fermat cubic in variables $w_i$ preserves the smoothness for every $r\geq 1$. Consider the projection 
    $$
    X\dashrightarrow \bP^{r+1},\quad (x_i,w_j)\to (x_4,x_5,w_1,\ldots,w_r).
    $$
    Let $K=k(\bP^{r+1})=k(a,b_1,\ldots,b_r)$. Set $a=x_4/x_5$, and $b_i=w_i/x_5$. Renaming the coordinates 
    $$
    Y_1=x_1,\quad Y_2=x_2, \quad Y_3=x_5,\quad Y_4=x_3,
    $$
 the generic fiber $X_\eta\subset\bP^3_{Y_1,Y_2,Y_3,Y_4}$ is given by
$$
Y_3(aY_1^2+2aY_1Y_2+(a^3+\beta)Y_3^2)+Y_4(Y_1^2+Y_1Y_2+Y_2^2-a^2Y_3^2+Y_4^2)=0,
$$
where $\beta=2+\sum_{i=1}^rb_i^3$. One can check that $X_\eta$ is smooth over $K$. Contracting the line $l_\eta=\{Y_3=Y_4=0\}\subset\bP^3$ yields a quartic del Pezzo surface $S$ birational to $X_\eta$ over $K$. Consider the  conic bundle obtained via projection from $l_\eta$
$$
\pi_\eta:X_\eta\to\bP^1,\quad (Y_1,Y_2,Y_3,Y_4)\to (Y_3,Y_4).
$$
Over the affine chart $t=Y_4/Y_3$, the generic fiber is the conic given by 
$$
f(Y_1,Y_2):=(t + a)Y_1^2 + (t + 2a)Y_1Y_2 + tY_2^2 + t^3 -a^2t + a^3 + \beta=0.
$$
The discriminant of the quadratic form $f$ is 
$$
\mathrm{disc}(f)=-\frac14(t^3-a^2t+a^3+\beta)(3t^2-4a^2).
$$
Let $c_1(t)=t^3-a^2t+a^3+\beta$, and $\rho_1,\rho_2,\rho_3$ its three roots. One sees that
$$
\Delta:=\disc(c_1)=-23a^6 - 54a^3\beta - 27\beta^2
$$
is not a square in $K$. Thus $\Gal(K_1/K)=\fS_3$, where $K_1=K(\rho_1,\rho_2,\rho_3)$. Put 
$$
d_1=\rho_2-\rho_3,\quad d_2=\rho_3-\rho_1,\quad d_3=\rho_1-\rho_2.
$$
The coefficients of $c_1(t)$ imply that 
$$
d_i^2=4a^2-3\rho_i^2.
$$
The singular fibers over $\rho_i$ split over $K_1$, and their components are
\begin{align}\label{eqn:I3fiber1}
F_i^\pm:    2(\rho_i+a)Y_1+(\rho_i+2a\pm d_i)Y_2=0,\quad i=1,2,3.
\end{align}
Let $c_2(t)=3t^2-4a^2$. The splitting field of $c_2$ is $K(\sqrt{3})$, and $c_2$ has two roots 
$\rho_4=\frac{2a}{\sqrt3}$ and $\rho_5=-\frac{2a}{\sqrt3}$. We check that $\sqrt{3}\notin K_1$. Indeed, $K(\sqrt 3)$ is different from the unique quadratic subextension $K(\sqrt{\Delta})/K$ of $K_1/K$, since ${\Delta/3}$ is  not a square in $K$. Set $K_2=K_1(\sqrt 3)$. It follows that
$$
\Gal(K_2/K)=\Gal(K(\sqrt 3)/K)\times \Gal(K_1/K)=\fC_2\times\fS_3.
$$
Over $K_2$, each singular fiber of $\pi_\eta$ is defined. Choose $D_1\in\bar K$ such that 
$$
D_1^2=(-32\sqrt3-52)a^4-(24\sqrt3+36)a\beta,
$$
and put 
\begin{align}
    \label{eqn:I3D2}
D_2=\frac{4a}{D_1}d_1d_2d_3.
\end{align}
One can check that 
\begin{align*}
D_2^2=(32\sqrt3-52)a^4+(24\sqrt3-36)a\beta.
\end{align*}
The singular fibers over $\rho_4$ and $\rho_5$ split over $K_2(D_1)$, with components: 
\begin{align}\label{eqn:I3fiber2}
    F_4^\pm&: 2a(3+2\sqrt{3} )Y_1+a(6+2\sqrt3)Y_2\pm{D_1}=0,\\\notag
    F_5^\pm&: 2a(3-2\sqrt{3} )Y_1+a(6-2\sqrt3)Y_2\pm{D_2}=0.
\end{align}
Thus, the splitting field of $X_\eta$ and $S$ is 
$$
K_2(D_1)=K(\sqrt3,\rho_1,\rho_2,\rho_3)(D_1)=K_1(D_1).
$$
We show that $D_1\notin K_2$. 
One can check that the minimal polynomial of $D_1$ over $K$ is 
\begin{align}
 \label{eqn:I3D1mini}   
 (t^2-D_1^2)(t^2-D_2^2)\in K[t].
\end{align}
Then $K(D_1)/K$ has degree 4 and is not Galois since $D_2\not\in K(D_1)$.
However, any index 4 subgroup in $\Gal(K_2/K)=\fC_2\times\fS_3$ is normal, which implies that $D_1\notin K_2$. It follows that 
$$
[K_1(D_1):K]=[K_2(D_1):K]=[K_2(D_1):K_2][K_2:K]=24.
$$ 

 Let $G=\Gal(K_1(D_1)/K)$. From~\eqref{eqn:I3D2}, one sees that $K_2(D_1)$ is the splitting field of \eqref{eqn:I3D1mini} over $K_1$. It follows that $K_1(D_1)/K_1$ is Galois, and $\Gal(K_1(D_1)/K_1)=\fC_2^2$, since it has 3 different quadratic subextensions $K_1(\sqrt3)$, $K_1(D_1+D_2)$ and $K_1(D_1-D_2)$. We determine the image of $\Gal(K_1(D_1)/K_1)$ in $W(\sD_5)$ by identifying its action on~\eqref{eqn:I3fiber2}.
\begin{itemize}
    \item One generator of $\Gal(K_1(D_1)/K_1)$ swaps $D_1$ and $D_2$, and maps to $(4,5)\in W(\sD_5)$.
    \item The other generator of $\Gal(K_1(D_1)/K_1)$ changes the signs of $D_1$ and $D_2$ simultaneously, and maps to $c_4c_5\in W(\sD_5).$
\end{itemize}
The subextension 
$$
K_1(D_1)\supset K_1\supset K
$$
yields an extension of Galois groups
\begin{align}
    \label{eqn:I3ext}
    1\to \fC_2^2\to G\to\fS_3\to 1.
\end{align}
The extension~\eqref{eqn:I3ext} splits since we know that 
$
  K_1\cap K(D_1)=K.  
$
Any element in $\Gal(K_1/K)=\fS_3$ can be lifted to $G$ as an element fixing $D_1$. We describe their actions on singular fibers~\eqref{eqn:I3fiber1} and~\eqref{eqn:I3fiber2}.
\begin{itemize}
    \item A lift of $(1,2,3)\in\fS_3$ to $G$ permutes $\rho_1\mapsto\rho_2\mapsto\rho_3$, and fixes $D_1$ and $D_2$. Its image in $W(\sD_5)$ is $(1,2,3)$.
    \item A lift $\tau$ of $(2,3)\in\fS_3$ fixes $\rho_1$ and $D_1$, and swaps $\rho_2$ and $\rho_3$. One can check that $\tau (d_1)=-d_1$, $\tau(d_2)=-d_3$, $\tau(d_3)=-d_2$, and thus $\tau(D_2)=-D_2$. Its image in $W(\sD_5)$ is $c_1c_2c_3c_5(2,3)$.
\end{itemize}
Thus the image of $\mathfrak g_K$ in $W(\sD_5)$ is 
$$
\langle (4,5),c_4c_5,(1,2,3),c_1c_2c_3c_5(2,3)\rangle,
$$
which is conjugate to the group of type $I_3$ listed in Proposition~\ref{prop:h1}. Applying Corollary~\ref{coro:dp4stab}, we conclude that $X$ is stably rational over $\bQ$. 
\end{proof}

\begin{rema}
 The cubic hypersurfaces in Propositions~\ref{prop:CubicI1} and~\ref{prop:cubictypeI3} are not isomorphic over $\bC$. One can see this by checking that the cubic surfaces $X\cap\{x_5=w_1=\ldots=w_r=0\}$
 in the two cases have different numbers of Eckardt points.
 
 It would be interesting to study quartic del Pezzo fibration structures of types $I_0,I_1,I_2,I_3$ on other nonrational varieties, such as singular quartic hypersurfaces or quartic double solids, cf. \cite{cheltsovquartic}, \cite{Voisinquartic}.
\end{rema}
\section{Stable linearizability}
\label{sect:C3C4}

Recall that a generically free regular $G$-action on a variety $X$ is called stably linearizable if $X\times\bP^r$ is equivariantly birational to the projectivization of a linear $G$-representation, with the trivial $G$-action on $\bP^r$, for some $r\geq 0$. Here, we give a detailed proof of stable linearizability of the unique 
$G$-minimal nonlinearizable del Pezzo surface satisfying Condition {\bf(H1)}, see \cite[Theorem 1.2]{Pro-2}.  We work over a field $k$ of characteristic zero containing $\zeta_3$, a primitive third root of 1. 

\begin{prop}
\label{prop:c3c4}
    Let $S$ be the del Pezzo surface of degree 4 given by 
    \begin{align*}
          S=\{x_1^2+\zeta_3x_2^2+\zeta_3^2x_3^2+x_4^2=x_1^2+\zeta_3^2x_2^2+\zeta_3x_3^2+x_5^2=0\},
    \end{align*}
    with an action of $G=\fC_3\rtimes\fC_4$ generated by 
    \begin{align*}
        \sigma:(\mathsf x)&\mapsto (x_2,x_3,x_1,\zeta_3x_4,\zeta_3^2x_5),\\
        \tau:(\mathsf x)&\mapsto (x_1,x_3,x_2,-x_5,x_4).
    \end{align*}
    Then the $G$-action on $S$ is stably linearizable. 
\end{prop}

A quick proof of this proposition, based on \cite{DR}, would go as follows: 
a $G$-variety is stably linearizable if and only if all twists (or equivalently, one {\em versal} twist) are stably rational, see \cite[Theorem 1.1]{DR}, combined with \cite[Theorem 2.5]{KT-notions}. The key is that 
$
S^G\neq \emptyset,
$
so that for every twist over a field $K/k$, the set of $K$-points of a form of $S$ is nonempty. Then Corollary~\ref{coro:stable} gives stable rationality over $K$. Nevertheless, we include a  direct proof within equivariant geometry, based on the equivariant version of the torsor formalism developed in \cite{HT-torsor} and \cite{KT-uni}.
\begin{proof}[Proof of Proposition~\ref{prop:c3c4}]
  Note that $k$ is the splitting field of $S$. Thus, over $k$, there is a unique universal torsor over $S$. 
By \cite[Corollary 9]{HT-torsor}, it suffices to show that \begin{enumerate}
 \item $\Pic(S)$ is a $G$-stably permutation module;
        \item the $G$-action lifts to the universal torsor $\cT$ of $S$;
        \item the $G$-action on $\cT$ is linearizable. 
    \end{enumerate}
   We have proved (1) in Lemma~\ref{lemm:stablyperm}. 
By \cite{HT-torsor}, (2) follows from the fact that $S^G\ne\varnothing$. To show (3), we repeat the Cox ring computation from Section~\ref{sect:rat} with the same notation, noting that $S$ is the blowup of $\bP^2$ at the following 5 points:
$$
p_1=[1:0:0],\quad p_2=[0:1:0],\quad p_3=[0:0:1], \quad 
p_4=[1:1:1],
$$
$$
p_5=[1:\zeta_3:\zeta_3^2].
$$
In the anticanonical model, the exceptional curves can be identified as follows
{\begin{align*}
E_1
&= \{
x_3-x_1+x_2=\;
x_4-\zeta_3^2x_1-x_2=\;
x_5-\zeta_3 x_1-x_2
=0\},\\
E_2
&= \{
x_3+x_1-x_2=\;
x_4+\zeta_3^2x_1+x_2=\;
x_5+\zeta_3 x_1+x_2
=0\},\\
E_3
&= \{
x_3-x_1-x_2=\;
x_4+\zeta_3^2x_1-x_2=\;
x_5+\zeta_3 x_1-x_2
=0\},\\
E_4
&= \{
x_3+x_1+x_2=\;
x_4-\zeta_3^2x_1+x_2=\;
x_5+\zeta_3 x_1-x_2
=0\},\\
E_5
&= \{
x_3+x_1+x_2=\;
x_4+\zeta_3^2x_1-x_2=\;
x_5-\zeta_3 x_1+x_2
=0\},\\
L_{12}
&= \{
x_3-x_1-x_2=\;
x_4-\zeta_3^2x_1+x_2=\;
x_5-\zeta_3 x_1+x_2
=0\},\\
L_{13}
&= \{
x_3+x_1-x_2=\;
x_4-\zeta_3^2x_1-x_2=\;
x_5-\zeta_3 x_1-x_2
=0\},\\
L_{14}
&= \{
x_3-x_1+x_2=\;
x_4+\zeta_3^2x_1+x_2=\;
x_5-\zeta_3 x_1-x_2
=0\},\\
L_{15}
&= \{
x_3-x_1+x_2=\;
x_4-\zeta_3^2x_1-x_2=\;
x_5+\zeta_3 x_1+x_2
=0\},\\
L_{23}
&= \{
x_3-x_1+x_2=\;
x_4+\zeta_3^2x_1+x_2=\;
x_5+\zeta_3 x_1+x_2
=0\},\\
L_{24}
&= \{
x_3+x_1-x_2=\;
x_4-\zeta_3^2x_1-x_2=\;
x_5+\zeta_3 x_1+x_2
=0\},\\
L_{25}
&= \{
x_3+x_1-x_2=\;
x_4+\zeta_3^2x_1+x_2=\;
x_5-\zeta_3 x_1-x_2
=0\},\\
L_{34}
&= \{
x_3-x_1-x_2=\;
x_4-\zeta_3^2x_1+x_2=\;
x_5+\zeta_3 x_1-x_2
=0\},\\
L_{35}
&= \{
x_3-x_1-x_2=\;
x_4+\zeta_3^2x_1-x_2=\;
x_5-\zeta_3 x_1+x_2
=0\},\\
L_{45}
&= \{
x_3+x_1+x_2=\;
x_4-\zeta_3^2x_1+x_2=\;
x_5-\zeta_3 x_1+x_2
=0\},\\
Q
&= \{
x_3+x_1+x_2=\;
x_4+\zeta_3^2x_1-x_2=\;
x_5+\zeta_3 x_1-x_2=0
\}.
\end{align*}}
 We present an explicit lift of the $G$-action to $\Cox(S)$, and verify that it is compatible with the $G$-action on $S$. Consider the $G$-action given by
\begin{align*}
\sigma(q)&=q,&
\sigma(e_1)&=e_2,&
\sigma(e_2)&=e_3,&
\sigma(e_3)&=e_1,\\
\sigma(e_4)&=e_4,&
\sigma(e_5)&=e_5,&
\sigma(\ell_{12})&=\zeta_3\ell_{23},&
\sigma(\ell_{13})&=\zeta_3\ell_{12},\\
\sigma(\ell_{14})&=-\zeta_3\ell_{24},&
\sigma(\ell_{15})&=-\ell_{25},&
\sigma(\ell_{23})&=\zeta_3\ell_{13},&
\sigma(\ell_{24})&=-\zeta_3\ell_{34},\\
\sigma(\ell_{25})&=-\ell_{35},&
\sigma(\ell_{34})&=\zeta_3\ell_{14},&
\sigma(\ell_{35})&=\ell_{15},&
\sigma(\ell_{45})&=\ell_{45},
\end{align*}
and
\begin{align*}
\tau(q)&=3(\zeta_3^2-1)e_4,&
\tau(e_1)&=\frac{\zeta_3^2-1}{3}\ell_{14},&
\tau(e_2)&=\frac{\zeta_3-\zeta_3^2}{3}\ell_{34},\\
\tau(e_3)&=\frac{\zeta_3-1}{3}\ell_{24},&
\tau(e_4)&=\frac{\zeta_3^2-1}{9}\ell_{45},&
\tau(e_5)&=\frac{\zeta_3-1}{9}q,\\
\tau(\ell_{12})&=\frac{\zeta_3^2-1}{3}\ell_{25},&
\tau(\ell_{13})&=\frac{\zeta_3-1}{3}\ell_{35},&
\tau(\ell_{14})&=(\zeta_3^2-\zeta_3)\ell_{23},\\
\tau(\ell_{15})&=(1-\zeta_3)e_1,&
\tau(\ell_{23})&=\frac{\zeta_3^2-\zeta_3}{3}\ell_{15},&
\tau(\ell_{24})&=(\zeta_3^2-1)\ell_{12},\\
\tau(\ell_{25})&=(\zeta_3-1)e_3,&
\tau(\ell_{34})&=(\zeta_3-1)\ell_{13},&
\tau(\ell_{35})&=(1-\zeta_3)e_2,\\
\tau(\ell_{45})&=3(\zeta_3-1)e_5.
\end{align*}
  One can check that
the ideal $\langle R_1,\ldots,R_{20}\rangle$ given in Section~\ref{sect:rat} is $G$-invariant, and the permutation action of $G$ on the generators agree with that on the exceptional curves listed above. In particular, the $G$-action on $\Cox(S)$ respects the grading by $\Pic(S)$.  Under the following basis of $\rH^0(S,-K_S)$:
\begin{align*}
s_1
&=q e_2e_3\ell_{23}
  -e_1\ell_{12}\ell_{13}\ell_{45},\\
s_2
&=\zeta_3 q e_1e_3\ell_{13}
  -\zeta_3^2 e_2\ell_{12}\ell_{23}\ell_{45},\\
s_3
&=q e_2e_3\ell_{23}
  +\zeta_3 q e_1e_3\ell_{13}
  +2\zeta_3^2 q e_1e_2\ell_{12}
  +\zeta_3^2 e_2\ell_{12}\ell_{23}\ell_{45}
  +e_1\ell_{12}\ell_{13}\ell_{45},\\
s_4
&=\zeta_3^2 q e_2e_3\ell_{23}
  -\zeta_3 q e_1e_3\ell_{13}
  -\zeta_3^2 e_2\ell_{12}\ell_{23}\ell_{45}
  +\zeta_3^2 e_1\ell_{12}\ell_{13}\ell_{45},\\
s_5
&=\zeta_3 q e_2e_3\ell_{23}
  -\zeta_3 q e_1e_3\ell_{13}
  -\zeta_3^2 e_2\ell_{12}\ell_{23}\ell_{45} 
  +\zeta_3 e_1\ell_{12}\ell_{13}\ell_{45},
\end{align*}
the resulting anticanonical map 
$$
\pi:\Spec(\Cox(S))\stackrel{|-K_S|}{\longdashrightarrow} S
$$
is equivariant under the prescribed $G$-action given in Proposition~\ref{prop:c3c4}. 

To show the linearizability of the $G$-action on $\cT$, it suffices to estabish that for the $G$-action on $\Spec(\Cox(S))$. 
The fixed locus $\bP(U)^G$ consists of four points. 
Consider the projection from the affine tangent space of one of these points, given by 
$$
\varphi:\Spec(\Cox(S))\dashrightarrow \bA^8
,\quad (q,e_1,\ldots,e_5,l_{ij})\mapsto (f_1,\ldots,f_8)$$ where, 
\begin{align*}
f_1&=e_2 - \zeta_3e_3 + (\zeta_3 - 1)e_4 + l_{12} - l_{13} + l_{14},\\
f_2&=\zeta_3e_2 - e_3 - (\zeta_3 - 1)e_5 + \zeta_3l_{12} -\zeta_3^2l_{13} + l_{15},\\
f_3&=e_1 -\zeta_3^2e_3 -(\zeta_3 + 2)e_4 + l_{12} - l_{23} + l_{24},\\
f_4&=e_1 - \zeta_3e_3 + (\zeta_3 - 1)e_5 + l_{12} -\zeta_3^2l_{23} + l_{25},\\
f_5&=\zeta_3e_1 -\zeta_3^2e_2 - (2\zeta_3 + 1)e_4 + l_{13} - l_{23} + l_{34},\\
f_6&=\zeta_3e_1 - e_2 - (\zeta_3 - 1)e_5 + l_{13} - \zeta_3l_{23} + l_{35},\\
f_7&=3e_1 - 3\zeta_3e_2 + 3(\zeta_3 - 1)e_5 - (\zeta_3 + 2)l_{14} + (2\zeta_3 + 1)l_{24} 
    + l_{45},\\
f_8&=q - 3(\zeta_3 +2)e_4 + 3l_{12} - 3\zeta_3l_{13}- (2\zeta_3 + 1)l_{25} + (\zeta_3 + 
    2)l_{35}.
\end{align*}
An explicit inverse of $\varphi$, given in \cite{TZdp4-tables},
proves that $\varphi$ is birational.
\end{proof}

\section{Arithmetic applications}
\label{sect:arithm}

J.-L. Colliot-Th\'el\`ene encouraged us to include the following consequences of Theorem~\ref{thm:mainrepeat}, i.e., the rationality of a universal torsor $\cT$ over a quartic del Pezzo surface $S$, provided $\cT(k)\neq \varnothing$:
\begin{enumerate}
    \item The evaluation map from $S(k)/R$ (the set of rational points modulo $R$-equivalence) to $\rH^1(k,T)$ is injective. 
 \item Let $A_{0}(X) \subset \mathrm{CH}_{0}(X)$ be the Chow group of zero-cycles of
degree zero modulo rational equivalence, and $p_0\in S(k)$.
  Then the map  
  \begin{align*}
      S(k)/R \to  A_{0}(S),\quad p\mapsto p-p_0
  \end{align*}
is bijective. 
Its surjectivity was proved in \cite{CTCoray}. 
\item Let $k$ be a finitely generated extension of $\bQ$. Then $S(k)/R$ is finite. This follows from (2) and the finiteness of $A_0(X)$, proved in \cite{CTchow} and \cite{Bloch}. 
\end{enumerate}
These were known when $S$ admits a conic bundle structure, by \cite{CTSkdP4R}.

	\bibliography{stablin}

\begin{thebibliography}{BCTSSD85}

\bibitem[BCTSSD85]{bctspstably}
A.~Beauville, J.-L. Colliot-Th\'el\`ene, J.-J. Sansuc, and P.~Swinnerton-Dyer.
\newblock Vari\'et\'es stablement rationnelles non rationnelles.
\newblock {\em Ann. of Math. (2)}, 121(2):283--318, 1985.

\bibitem[Blo81]{Bloch}
S.~Bloch.
\newblock On the {C}how groups of certain rational surfaces.
\newblock {\em Ann. Sci. \'Ecole Norm. Sup. (4)}, 14(1):41--59, 1981.

\bibitem[BP04]{batycox}
V.~V. Batyrev and O.~N. Popov.
\newblock The {C}ox ring of a del {P}ezzo surface.
\newblock In {\em Arithmetic of higher-dimensional algebraic varieties ({P}alo
  {A}lto, {CA}, 2002)}, volume 226 of {\em Progr. Math.}, pages 85--103.
  Birkh\"auser, 2004.

\bibitem[BW23]{BW}
O.~Benoist and O.~Wittenberg.
\newblock Intermediate {J}acobians and rationality over arbitrary fields.
\newblock {\em Ann. Sci. \'Ec. Norm. Sup\'er. (4)}, 56(4):1029--1084, 2023.

\bibitem[CG72]{CG}
C.~H. Clemens and P.~A. Griffiths.
\newblock The intermediate {J}acobian of the cubic threefold.
\newblock {\em Ann. of Math. (2)}, 95:281--356, 1972.

\bibitem[Che06]{cheltsovquartic}
I.~Cheltsov.
\newblock Nonrational nodal quartic threefolds.
\newblock {\em Pacific J. Math.}, 226(1):65--81, 2006.

\bibitem[CMR04]{CMR}
C.~Ciliberto, M.~Mella, and F.~Russo.
\newblock Varieties with one apparent double point.
\newblock {\em J. Algebraic Geom.}, 13(3):475--512, 2004.

\bibitem[CR06]{projection}
C.~Ciliberto and F.~Russo.
\newblock Varieties with minimal secant degree and linear systems of maximal
  dimension on surfaces.
\newblock {\em Adv. Math.}, 200(1):1--50, 2006.

\bibitem[CTC79]{CTCoray}
J.-L. Colliot-Th\'el\`ene and D.~Coray.
\newblock L'\'equivalence rationnelle sur les points ferm\'es des surfaces
  rationnelles fibr\'ees en coniques.
\newblock {\em Compositio Math.}, 39(3):301--332, 1979.

\bibitem[CTS77]{CS-tori}
J.-L. Colliot-Th\'el\`ene and J.-J. Sansuc.
\newblock La {$R$}-\'equivalence sur les tores.
\newblock {\em Ann. Sci. \'Ecole Norm. Sup. (4)}, 10(2):175--229, 1977.

\bibitem[CTS81]{CTchow}
J.-L. Colliot-Th\'el\`ene and J.-J. Sansuc.
\newblock On the {C}how groups of certain rational surfaces: a sequel to a
  paper of {S}. {B}loch.
\newblock {\em Duke Math. J.}, 48(2):421--447, 1981.

\bibitem[CTS87a]{CTSansucDuke}
J.-L. Colliot-Th\'{e}l\`ene and J.-J. Sansuc.
\newblock La descente sur les vari\'{e}t\'{e}s rationnelles. {II}.
\newblock {\em Duke Math. J.}, 54(2):375--492, 1987.

\bibitem[CTS87b]{CTSkdP4R}
J.-L. Colliot-Th\'el\`ene and A.~N. Skorobogatov.
\newblock {$R$}-equivalence on conic bundles of degree {$4$}.
\newblock {\em Duke Math. J.}, 54(2):671--677, 1987.

\bibitem[Der07]{Coxderenthalhomogeneous}
U.~Derenthal.
\newblock Universal torsors of del {P}ezzo surfaces and homogeneous spaces.
\newblock {\em Adv. Math.}, 213(2):849--864, 2007.

\bibitem[DP19]{derenthal-non}
U.~Derenthal and M.~Pieropan.
\newblock Cox rings over nonclosed fields.
\newblock {\em J. Lond. Math. Soc. (2)}, 99(2):447--476, 2019.

\bibitem[DR15]{DR}
A.~Duncan and Z.~Reichstein.
\newblock Versality of algebraic group actions and rational points on twisted
  varieties.
\newblock {\em J. Algebr. Geom.}, 24(3):499--530, 2015.

\bibitem[EFS25]{cubicmatroid}
P.~Engel, O.d.G. Fortman, and S.~Schreieder.
\newblock Matroids and the integral {H}odge conjecture for abelian varieties,
  2025.
\newblock {\tt arXiv:2507.15704}.

\bibitem[HKT22]{HKT-quad}
B.~Hassett, J.~Koll\'ar, and Yu. Tschinkel.
\newblock Rationality of even-dimensional intersections of two real quadrics.
\newblock {\em Comment. Math. Helv.}, 97(1):183--207, 2022.

\bibitem[HKT23]{HKT-duke}
B.~Hassett, A.~Kresch, and Yu. Tschinkel.
\newblock Stable rationality in smooth families of threefolds.
\newblock {\em Duke Math. J.}, 172(6):1145--1172, 2023.

\bibitem[HT17a]{HT-1}
B.~Hassett and Yu. Tschinkel.
\newblock Stable rationality of quartic del {P}ezzo surfaces, 2017.
\newblock unpublished note.

\bibitem[HT17b]{HT-2}
B.~Hassett and Yu. Tschinkel.
\newblock Stably rational cubic threefolds, 2017.
\newblock unpublished note.

\bibitem[HT21]{HT-quad}
B.~Hassett and Yu. Tschinkel.
\newblock Rationality of complete intersections of two quadrics over nonclosed
  fields.
\newblock {\em Enseign. Math.}, 67(1-2):1--44, 2021.
\newblock With an appendix by Jean-Louis Colliot-Th\'{e}l\`ene.

\bibitem[HT23]{HT-torsor}
B.~Hassett and Yu. Tschinkel.
\newblock Torsors and stable equivariant birational geometry.
\newblock {\em Nagoya Math. J.}, 250:275--297, 2023.

\bibitem[KiST89]{KSS}
B.~Kunyavski\u~i, A.~Skorobogatov, and M.~Tsfasman.
\newblock del {P}ezzo surfaces of degree four.
\newblock {\em M\'em. Soc. Math. France (N.S.)}, (37):113, 1989.

\bibitem[KT25]{KT-uni}
A.~Kresch and Yu. Tschinkel.
\newblock Equivariant unirationality of toric varieties, 2025.
\newblock {\tt arXiv:2506.07152}.

\bibitem[KT26a]{KT-survey}
A.~Kresch and Yu. Tschinkel.
\newblock Invariants in equivariant birational geometry, 2026.
\newblock {\tt arXiv:2602.23998}.

\bibitem[KT26b]{KT-notions}
A.~Kresch and Yu. Tschinkel.
\newblock Linearizability notions in equivariant birational geometry, 2026.
\newblock {\tt arXiv:2606.10965}.

\bibitem[Man66]{Maninperfect}
Yu. \~I. Manin.
\newblock Rational surfaces over perfect fields.
\newblock {\em Inst. Hautes \'Etudes Sci. Publ. Math.}, (30):55--113, 1966.

\bibitem[Man86]{manin-book}
Yu. \~I. Manin.
\newblock {\em Cubic forms}, volume~4 of {\em North-Holland Mathematical
  Library}.
\newblock North-Holland Publishing Co., Amsterdam, second edition, 1986.

\bibitem[Pro15]{Pro-2}
Yu. Prokhorov.
\newblock On stable conjugacy of finite subgroups of the plane {C}remona group,
  {II}.
\newblock {\em Michigan Math. J.}, 64(2):293--318, 2015.

\bibitem[SB04]{SBstab}
N.~I. Shepherd-Barron.
\newblock Stably rational irrational varieties.
\newblock In {\em The {F}ano {C}onference}, pages 693--700. Univ. Torino,
  Turin, 2004.

\bibitem[SS91]{Salberger}
P.~Salberger and A.~N. Skorobogatov.
\newblock Weak approximation for surfaces defined by two quadratic forms.
\newblock {\em Duke Math. J.}, 63(2):517--536, 1991.

\bibitem[SS07]{SS}
V.~Serganova and A.~Skorobogatov.
\newblock Del {P}ezzo surfaces and representation theory.
\newblock {\em Algebra Number Theory}, 1(4):393--419, 2007.

\bibitem[TY20]{TY}
Yu. Tschinkel and K.~Yang.
\newblock Potentially stably rational del {P}ezzo surfaces over nonclosed
  fields.
\newblock In {\em Combinatorial and additive number theory. {III}}, volume 297
  of {\em Springer Proc. Math. Stat.}, pages 227--233. Springer, Cham, 2020.

\bibitem[TZ26]{TZdp4-tables}
Yu. Tschinkel and Zh. Zhang.
\newblock Supporting materials, 2026.
\newblock available at: {\tt https://zhijiazhangz.github.io/dP4/}.

\bibitem[Voi15]{Voisinquartic}
C.~Voisin.
\newblock Unirational threefolds with no universal codimension {$2$} cycle.
\newblock {\em Invent. Math.}, 201(1):207--237, 2015.

\bibitem[Voi17]{Voisincubic}
C.~Voisin.
\newblock On the universal {$\rm CH_0$} group of cubic hypersurfaces.
\newblock {\em J. Eur. Math. Soc. (JEMS)}, 19(6):1619--1653, 2017.

\end{thebibliography}
	\bibliographystyle{alpha}
	
\end{document}